\documentclass{article}

\usepackage{amsmath}
\usepackage{amssymb}
\usepackage{amsthm}
\usepackage{graphicx}
\usepackage{amscd}
\usepackage{epic, eepic}
\usepackage{url}
\usepackage{color}
\usepackage[utf8]{inputenc} 
\usepackage{enumerate}   
\usepackage{graphicx}
\usepackage{epstopdf}

\newtheorem{theorem}{\bf Theorem}[section]

\newtheorem{conj}[theorem]{\bf Conjecture}

\newtheorem{claim}[theorem]{\bf Claim}
\theoremstyle{definition}

\newtheorem{defi}[theorem]{\bf Definition}

\DeclareMathOperator{\ex}{ex}

\title{The Turán number of blow-ups of trees}
\author{Andrzej Grzesik\thanks{Faculty of Mathematics and Computer Science, Jagiellonian University, Łojasiewicza 6, 30-348 Krak\'{o}w, Poland. E-mail: {\tt Andrzej.Grzesik@uj.edu.pl}. The work of this author has received funding from the European Research Council (ERC) under the European Union’s Horizon 2020 research and innovation programme (grant agreement No~648509). This publication reflects only its authors' view; the European Research Council Executive Agency is not responsible for any use that may be made of the information it contains.} \and 
	Oliver Janzer\thanks{Department of Pure Mathematics and Mathematical Statistics, University of Cambridge, United Kingdom.
		E-mail: {\tt oj224@cam.ac.uk}.} \and
	Zoltán Lóránt Nagy\thanks{MTA--ELTE Geometric and Algebraic Combinatorics Research Group,
  E\"otv\"os Lor\'and University, Budapest, Hungary. The author is supported by the Hungarian Research Grant (NKFI) No. K 120154 and by the ÚNKP-18-4 New National Excellence Program  of the Ministry of Human Capacities(Bolyai+). 	E-mail: {\tt nagyzoli@cs.elte.hu}}} 
\date{}

\begin{document}

\maketitle

\begin{abstract}
%A graph is $r$-degenerate if each of its subgraphs has minimum degree at most $r$. 
A conjecture of Erdős from 1967 asserts that any graph on $n$ vertices which does not contain a fixed $r$-degenerate bipartite graph $F$ has at most $Cn^{2-1/r}$ edges, where $C$ is a constant depending only on $F$. We show that this bound holds for a large family of \hbox{$r$-degenerate} bipartite graphs, including all $r$-degenerate blow-ups of trees. Our results generalise many previously proven cases of the Erd\H{o}s conjecture, including the related results of Füredi and Alon, Krivelevich and Sudakov. Our proof uses supersaturation and a random walk on an auxiliary graph.%relies on a probabilistic argument, closely related to Markov-chains.
%ZLN: This sounds really promising and convincing, but I think this is not true in fact.
%AG: I don't know which result we are not covering, but I changed "all" to "many".

\end{abstract}

\section{Introduction}
For a simple graph $F$, the Tur\'an number  $\ex(n, F)$ is defined as the maximum number of edges in an $n$-vertex simple graph not containing $F$ as a subgraph. While this function is well understood for graphs with chromatic number $\chi(F)$ larger than $2$, in case of bipartite graphs $F$ not even the order of magnitude is known in general. We refer to the detailed survey of F\"uredi and Simonovits~\cite{Furedi-Simonovits} on the subject. 

The K\H{o}v\'ari-S\'os-Tur\'an theorem \cite{KST} states that if $F$ is the complete bipartite graph $K_{r,t}$ for $r\le t$, then $\ex(n, F) = O(n^{2-1/r})$. %Erd\H{o}s presume a strong connection between the degeneracy on the bipartite graph and the exponent of its Tur\'an function.
A graph $F$ is called {\em $r$-degenerate} if each of its subgraphs has minimum degree at most $r$.
Generalising the K\H{o}v\'ari-S\'os-Tur\'an theorem, Erd\H{o}s in 1967 proposed the following conjecture. 

\begin{conj}[Erdős \cite{Erdos}]\label{Erd-conj}
Let $F$ be a bipartite  $r$-degenerate graph. Then  $\ex(n, F)=O(n^{2-\frac{1}{r}})$.
\end{conj}
Note that this conjecture would be tight due to the results of Alon, Rónyai and Szabó~\cite{ARSz} and Kollár, Rónyai and Szabó~\cite{KRSz} on the Turán number of complete bipartite graphs $K_{r,s}$, where $s>(r-1)!$.

The first partial result towards this conjecture which also proved a weaker conjecture of Erdős was obtained by Füredi. In fact, this was only implicit in \cite{Furedi}.

\begin{theorem}[Füredi \cite{Furedi}]\label{semireg}
Let $F$ be bipartite graph with maximum degree at most $r$ on one side. Then there exists a constant $C$ depending only on $F$ for which $\ex(n, F)\leq Cn^{2-\frac{1}{r}}$.
\end{theorem}

This was reproved using the celebrated {\em dependent random choice} method by Alon, Krivelevich and Sudakov \cite{Alon}, see also \cite{Fox}.
%The method of Alon, Krivelevich and Sudakov implies the following  general
They used their techniques to prove the following result as well, which provides a general but weaker bound on the Turán function than Conjecture~\ref{Erd-conj}.

\begin{theorem}[Alon, Krivelevich, Sudakov \cite{Alon}]
Let $F$ be a bipartite  $r$-degenerate graph. Then  $\ex(n, F)=O(n^{2-\frac{1}{4r}})$.
\end{theorem}

Recently, Conlon and Lee \cite{Conlon}, and the second author \cite{JanzerO} improved Theorem \ref{semireg} when $r=2$, showing that the exponent is always smaller than $2-1/2$ except when $F$ contains the complete bipartite graph $K_{2,2}$ as a subgraph. They studied the Turán function of the subdivisions of complete graphs on at least $3$ vertices. Note that any $K_{2,2}$-free bipartite graph with maximum degree at most $2$ on one side is a subgraph of the subdivision of a sufficiently large complete graph.
%$\{H_t: t>2\}$ of the complete graphs $\{K_t : t>2\}$. 
%Observe that these graphs  $2$-degenerate bipartite graphs which does not contain an induced $K_{2,2}$
%\todo{OJ: I think this isn't in fact true. Do we mean bipartite graphs with max degree at most 2 on one side rather than 2-degenerate?} \todo{ZLN:For sure. Probably the whole sentence is not so relevant}
%one of the graphs of the graph family $\{H_t: t>2\}$.
The second author proved that the Tur\'an function of the subdivision of $K_t$ is $O(n^{3/2-\frac{1}{4t-6}})$, which is tight for $t=3$.
%The second author obtained a remarkable reduction on the exponent compared to $3/2$ in terms of $t$ which happens to be tight for $t=3$. \todo{What is $t$ here?}

Concerning the case $r\geq 2$, another type of extension is due to Füredi and West \cite{Furedi-West}, who confirmed  $\ex(K_{s,s}\setminus K_{s-r, s-r})= O(n^{2-1/r})$, yet another weaker conjecture of Erdős, along their proof of a Ramsey-type result.  Here the forbidden graph is obtained from the complete bipartite graph $K_{s,s}$ by deleting a complete bipartite subgraph $ K_{s-r, s-r}$.

Observe that there exists a permutation of the vertices $\{v_1, v_2, \ldots, v_k\}$ of any $r$-degenerate graph~$F$, for which every vertex $v_i$ has at most $r$ neighbours in the set  $\{v_1, v_2, \ldots, v_{i-1}\}$. With this in mind, one can define the complexity  of an $r$-degenerate graph as follows.
\begin{defi} {
The graph $K_{r,r}=G(A_0, B_0)$ is considered as a graph of complexity $0$ and any multiplicity. A bipartite graph $G(A, B)$ is a \emph{complete $r$-degenerate bipartite graph of complexity $s$ and multiplicity $m$} if it can be obtained from the complete bipartite graph $G(A', B')$ of complexity $s-1$ and multiplicity $m$ by the addition of further $m(\binom{|A'|}{r}+\binom{|B'|}{r})$ vertices such that $m$ new vertices are assigned to each $r$-set in $A'$ and each $r$-set in $B'$, and every new vertex is connected to the vertices of the $r$-set that it is assigned to.
%This definition can be reformulated as an {\em iterated subdivision} as well in the case of $r=2$, and the case $r>2$ is a generalized subdivision.
The \emph{complexity} of an $r$-degenerate bipartite graph~$F$ is defined to be the smallest possible complexity of a complete $r$-degenerate bipartite graph (of arbitrary multiplicity) that contains $F$ as a subgraph. }
%\todo[inline]{OJ: I think this last paragraph is not quite rigorous. Indeed, there is a say 2-degenerate graph that is contained in complexity 1, multiplicity 2 and also complexity 2 and multiplicity 1, but not complexity 1, multiplicity 1. My point is that we can't uniquely define both complexity and multiplicity of a graph at the same time. I suggest we define only the complexity of a graph (to be the smallest possible complexity with arbitrary multiplicity)}
\end{defi}

\begin{figure}[h!]
\centering
  \includegraphics[width=13cm]{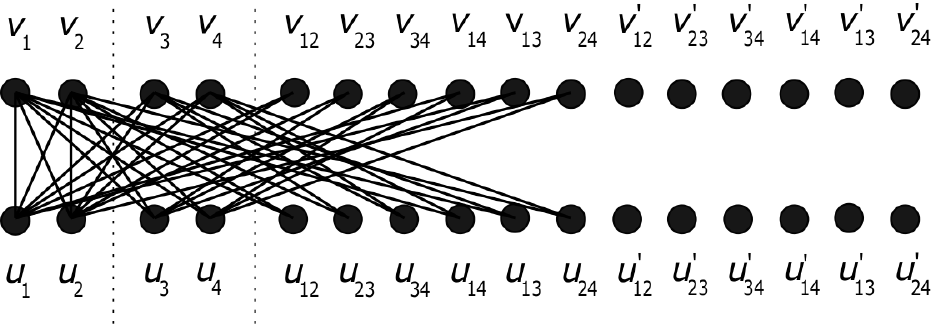}
  
  \caption{\label{complex} The complete $2$-degenerate bipartite graph of complexity $2$ and multiplicity $2$. Note that the clone $v'$ of $v$ has the same neighbours, but we did not draw those edges in order to keep the figure transparent.}

 \end{figure}
 
\begin{figure}[h!]
\centering
  \includegraphics[width=13cm]{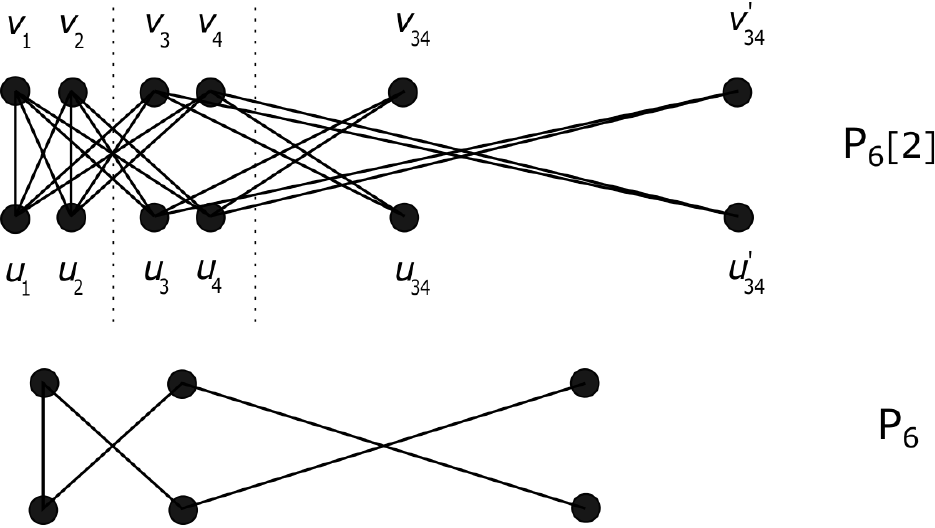}
  
  \caption{\label{path} The blow-up $P_6[2]$ of a path on $6$ vertices, as a subgraph of the complete $2$-degenerate bipartite graph of complexity $2$ and multiplicity $2$.}

 \end{figure}

%\todo[inline]{ZLN: I plan to add a nice couple of pictures here: one for case $r=3$, $s=3$, and one with its subgraph, which is a blownup of a path with the corresponding complexity and same $r$}

Note that the result of F\"uredi and West covers precisely the complexity $1$ case, while Theorem \ref{semireg} only applies to some $r$-degenerate bipartite graphs of complexity at most $2$.

%Theorem \ref{semireg} of Füredi and a result of Alon, Krivelevich and Sudakov \cite{Alon} is essentially corresponding to the graphs contained in bipartite graphs of complexity at most $2$, while the Füredi and West result \cite{Furedi-West} is exactly the complexity $1$ case.
 
Our first contribution is a proof of Conjecture \ref{Erd-conj} for all graphs of complexity at most $2$.

\begin{theorem} \label{complexity2}
    Let $F$ be a complete $r$-degenerate bipartite graph of complexity $2$ and arbitrary multiplicity. Then $$\ex(n,F)=O(n^{2-\frac{1}{r}}).$$
\end{theorem}

Our next result concerns the case where  $F$ has larger complexity but has a strong structure, namely where $F$ is a blow-up of a tree, see Figure \ref{path}.

\begin{theorem}\label{main1} Let $T$ denote a tree and let $T[r]$ denote its blow-up, where every vertex is replaced by an independent set of $r$ vertices, and the copies of two vertices are adjacent if and only if the originals are. Then $$\ex(n, T[r])=O (n^{2-\frac{1}{r}}).$$ 
\end{theorem}

Actually, the vertices can be replaced by sets of arbitrary sizes as long as the resulting graph is $r$-degenerate, and the same conclusion holds. We say that a graph $F$ is a blow-up of the graph~$T$ if to get $F$ from $T$ we replace each vertex of $T$ with an independent set (of arbitrary size) and replace each edge of $T$ with a complete bipartite graph. 
%The leaves can be replaced by larger sets in Theorem \ref{main1}, and the same conclusion holds.

\begin{theorem} \label{main4}
    Let $F$ be a graph that is $r$-degenerate and is a blow-up of a tree. Then $$\ex(n,F)=O(n^{2-\frac{1}{r}}).$$
\end{theorem}
%\begin{theorem}\label{main3} Consider the blow-up $T[r,t]$ of a tree $T$ on at least $3$ vertices, where every leaf is replaced by an independent set of $t\ge r$ vertices and every non-leaf is replaced by an independent set of $r$ vertices. Then $$\ex(n, T[r, t])=O(n^{2-\frac{1}{r}}).$$ 
%\end{theorem}

Note that Theorem \ref{main4} is a generalisation of the result of Füredi and West \cite{Furedi-West} on the Tur\'an number $\ex(n, K_{s,s}\setminus K_{s-r, s-r})$. This case corresponds to the blow-up of the path $P_4$. 

In fact, we prove an even more general statement from which Theorem \ref{main4} follows. To state this result, we need to introduce another definition.

\begin{defi}\label{rtblownup}{Let $r\leq t$ and $k$ be positive integers and let $X_1=Y_0,Y_1,Y_2\ldots,Y_k$ be pairwise disjoint sets with $|X_1|=r,|Y_1|=\ldots=|Y_k|=t$. For each $2\leq i\leq k$, let $X_i$ be a subset of some $Y_j$ with $j<i$ such that $|X_i|=r$.
The graph $L$ with vertex set $Y_0\cup Y_1\cup \ldots \cup Y_k$ and edge set $\bigcup_{1\leq i\leq k} \{xy: x\in X_i,y\in Y_i\}$ is called an \emph{$(r,t)$-blownup tree of size $k$}}.
\end{defi}
    
See Figure \ref{rtfigure} for an example of a $(2,3)$-blownup tree of size $4$. 
    
%An $(r,r^*)$-blownup tree of size $1$ is just a $K_{r,r^*}$ and its only core $r$-set and only core $r^*$-set are the parts of the $K_{r,r^*}$. For $k\geq 2$, an $(r,r^*)$-blownup tree of size $k$, denoted $L$, is obtained as follows. Take an $(r,r^*)$-blownup tree of size $k-1$, denoted $L_-$. Take a subset $R$ of size $r$ of some core $r^*$-set of $L_-$, and a set $R^*$ of $r^*$ vertices disjoint from $V(L_-)$. Now set $V(L)=V(L_-)\cup R^*$ and $E(L)=E(L_-)\cup \{xy: x\in R,y\in R^*\}$. Finally, let the core $r$-sets of $L$ be the core $r$-sets of $L_-$ together with $R$, and let the core $r^*$-sets of $L$ be the core $r^*$-sets of $L_-$ together with $R^*$.

%\todo[inline]{It would be good to add some picture}

\begin{figure}[h!]
\centering
  \includegraphics[width=10cm]{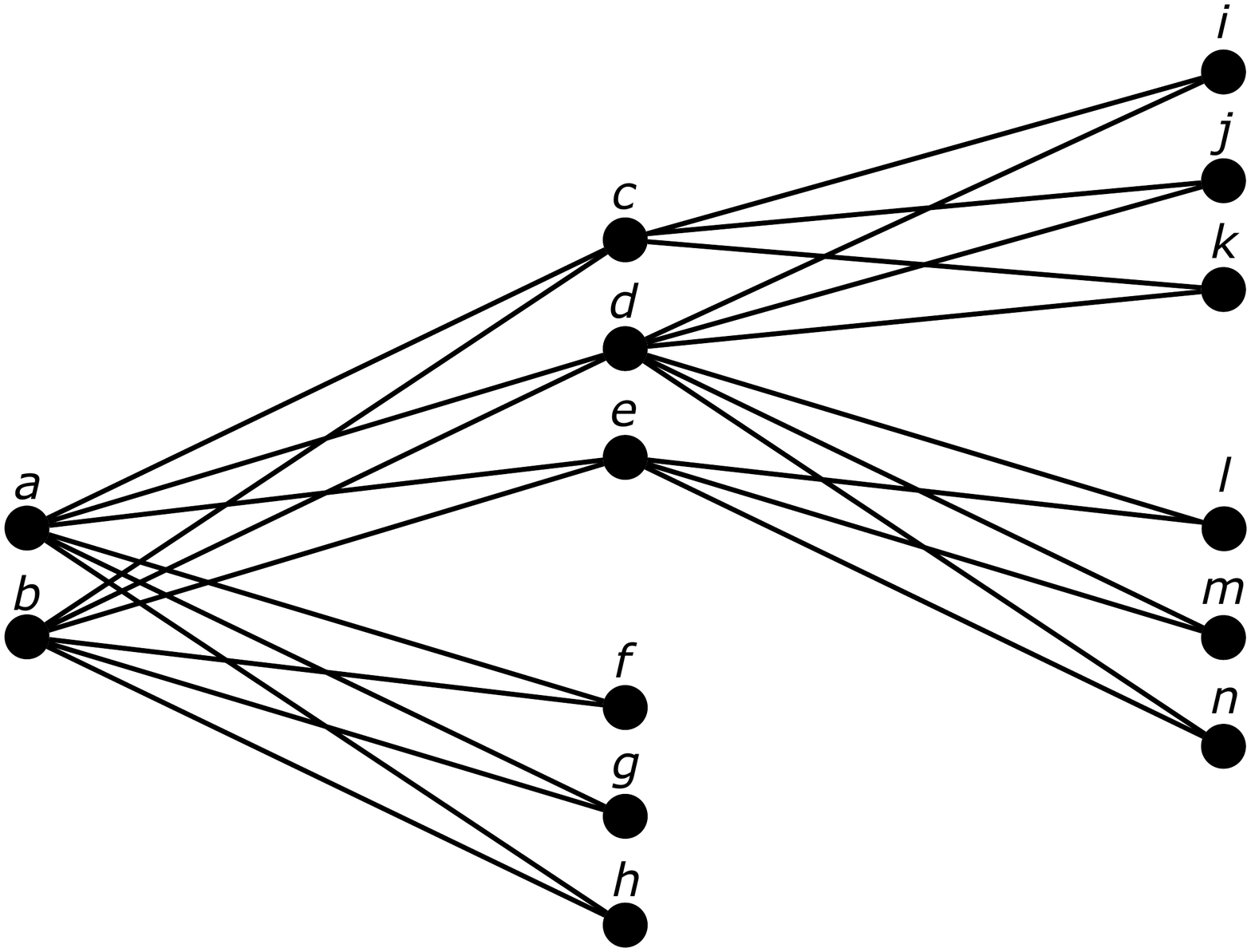}
  
  \caption{\label{rtfigure} A $(2,3)$-blownup tree of size $4$. Here $X_1=Y_0=\{a,b\}$, $X_2=\{a,b\}$, $X_3=\{c,d\}$, $X_4=\{d,e\}$, $Y_1=\{c,d,e\}$, $Y_2=\{f,g,h\}$, $Y_3=\{i,j,k\}$, $Y_4=\{l,m,n\}$.      }

 \end{figure}

%Informally, we obtain the new graph by adding a new $K_{r,r^*}$ and gluing it to the previous graph along an $r$-subset of a previous $r^*$-set and the $r$-part of the $K_{r,r^*}$.

Observe that an $(r,t)$-blownup tree is $r$-degenerate. We are now ready to state our most general result.

\begin{theorem} \label{main2}
    Let $L$ be an $(r,t)$-blownup tree of arbitrary size. Then $$\ex(n,L)=O(n^{2-\frac{1}{r}}).$$
\end{theorem}

%Now observe that if $T_k$ is a tree on $k$ vertices, then $T_k\lbrack r\rbrack$ is an $(r,r)$-blownup tree of size $k-1$. Hence, Theorem \ref{main2} implies Theorem \ref{main1}. Another corollary of Theorem \ref{main2} is given as follows. 

%Given this result, it is easy to deduce Theorem \ref{main3}.

%\begin{proof}[Proof of Theorem \ref{main3}]
%    Note that $T\lbrack r,t\rbrack$ for a tree $T$ on $k$ vertices is a subgraph of a certain $(r,t)$-blownup tree of size $k-1$, so the result follows from Theorem \ref{main2}.
%\end{proof}

Note that any bipartite graph $F$ with maximum degree at most $r$ on one side is a subgraph of some $(r,t)$-blownup tree (for a suitable $t$). Indeed, when the parts of $F$ are $X$ and $Y$ such that every vertex in $X$ has degree at most $r$, then $t$ can be chosen to be $|Y|$. This shows that Theorem~\ref{main2} generalises Theorem \ref{semireg}.

The rest of the paper is organised as follows. In Section \ref{sectionthm18} we present the proofs of Theorem \ref{complexity2}, Theorem \ref{main4} and Theorem \ref{main2}, while in Section \ref{sectionfurtherblowup} we discuss further generalisations and related problems.

Let us briefly summarise the method we will use in Section \ref{sectionthm18}. Roughly speaking, we prove that if we randomly and greedily try to embed an $(r,t)$-blownup tree $L$ in the host graph, then with positive probability we do not get stuck. The way we choose the embedded images of the first few vertices of $L$ is not straightforward: we make use of the stationary distribution on an auxiliary graph whose vertices are the $r$-sets of the original host graph. To obtain a dense enough auxiliary graph, we apply results on graph supersaturation. 
The embedding of the further vertices is also closely related to the usual random walk on this auxiliary graph, which allows us to prove that with high probability all $r$-sets that we hit in the random embedding have large enough neighbourhood.

%\todo[inline]{ZLN: Here might come a short description of the proof technique: auxiliary graph and derandomized embedding via Markov-chains. Something similar happens in \cite{Conlon2} \\OJ: Are you happy with this description?\\ZLN: yes :) (tiny additions were made)}

\section{The proofs} \label{sectionthm18}

For a graph $G$, $\overline{d}(G)$ denotes its average degree, $N(v)=N_G(v)$ denotes the set of neighbours of vertex $v$, while the common neighbourhood of a certain vertex set $R$ is denoted by $N(R)=N_G(R)$. We write $d_G(v)=|N_G(v)|$ and $d_G(R)=|N_G(R)|$. 
%The support $supp(\mathcal{A})$ of a family of sets $\mathcal{A}=~\{A_1, A_2, \ldots A_s\}$ consists of all their elements, i.e., $supp(\mathcal{A})=\bigcup_{i=1}^{s}A_i$. 
We call a set of $r$ vertices an $r$-set.

One of the main ingredients of the proofs is a theorem on supersaturated graphs.
Theorems of supersaturation are not only interesting on their own but their application can directly lead to further extremal results. Earlier examples in this direction are due to Füredi \cite{Furedi-cube} on $\ex(n, Q_8)$ and to Erdős and Simonovits \cite{E-Si} on the Turán number $\ex(n, \{C_4, C_5\})$, see also \cite{Sim84}.
We recall the version concerning complete bipartite graphs.

\begin{theorem}[Erdős, Simonovits \cite{Erdos, Furedi-Simonovits}]\label{supersatu} For any positive integers $r\leq  t$  and a real number $\gamma > 0$ there exists a constant $c=c_{r,t}(\gamma)$  such that  any graph on $n$ vertices with $e>c\cdot n^{2-\frac{1}{r}}$ edges contains at least $\gamma \binom{n}{r}$ copies of $K_{r,t}$.
\end{theorem}

We only need the former weaker version, but in its full strength, the theorem states that the number of copies is bounded from below by $\gamma' \frac{e^{rt}}{n^{2rt-r-t}}$ with an appropriate  $\gamma'$ provided that $e$ is much larger than the Turán function of $K_{r,t}$. We also note that the connection between $c_{r,t}(\gamma)$ and $\gamma$ is approximately $\gamma\approx \binom{(c/2)^r}{t}$ if $(c/2)^r>t$.

The proof relies on a convexity argument (or Jensen's inequality), and random bipartite graphs show that it is tight up to a constant factor. Note that   $n^{2-\frac{1}{r}}$  is the order of magnitude of the Turán function of $K_{r,t}$. In special cases, the supersaturation is even more understood  when the edge cardinality is in the interval $[\ex(n,K_{r,t}), (1+\varepsilon)\ex(n,K_{r,t})]$, see the paper of 
 the third author~\cite{ZNagy} for exact results in the case $r=t=2$ and on the dependence of $c_{r,t}(\gamma)$ on $\gamma$.

%Our second tool is an elementary lemma, see e.g. \cite{Furedi-Simonovits}.

%\begin{lemma}\label{mindeg} For all graphs $G$ there exists a subgraph $G'\subseteq G$ such that $d_{min}(G')\geq \frac{1}{2}\overline{d}(G)$.
%\end{lemma}

%We assign an auxiliary graph $\mathcal{G}$ to $G$ as follows. The vertices of $\mathcal{G}$ are the $r$-sets in $V(G)$, and two such $r$-sets $X$ and $Y$ are joined by an edge in $\mathcal{G}$ if $xy\in E(G)$ for every $x\in X$ and $y\in Y$. Our main tool for proving Theorem \ref{main2} is the so-called $T$-branching random walk, introduced by Conlon, Kim, Lee and Lee \cite{CKLL18}. We investigate this random walk on the auxiliary graph $\mathcal{G}$.

%Let $T$ be a tree with a fixed root $x_0\in V(T)$ and let $H$ be a graph with at least one edge. For any $x\in V(T)\setminus \{x_0\}$, the \emph{parent} of $x$ in $T$ is the unique neighbour of $x$ in the unique path from $x$ to $x_0$. Take an ordering $x_0,x_1,\ldots,x_k$ of the vertices of $T$ with the property that for every $j>0$, the parent of $x_j$ is $x_i$ for some $i<j$. We define a (not necessary uniformly) random homomorphism $\theta: T\rightarrow H$ as follows. For each $v\in V(H)$, let $\theta(x_0)=v$ with probability $\frac{d_H(v)}{2e(H)}$ (this is called the stationary distribution on $H$). Let $j>0$ and suppose that $\theta(x_a)$ has been defined for all $a<j$. Let $x_i$ be the parent of $x_j$. Then let $\theta(x_j)$ be a uniformly random neighbour of $\theta(x_i)$ in $H$. This procedure defines $\theta$ on $\{x_0,\ldots,x_k\}=V(T)$ and gives a random homomorphism $T\rightarrow H$.

We start with the proof of Theorem \ref{complexity2} which is simpler but already contains some of the ideas needed in the proof of Theorem \ref{main2}.

\begin{proof}[Proof of Theorem \ref{complexity2}]
    Let $m$ be the multiplicity of $F$. It is not hard to see that it suffices to find distinct vertices $u_1,u_2,\ldots,u_{r+m}$ and $v_1,v_2,\ldots,v_{r+m}$ in $V(G)$ such that
    \begin{enumerate}[(i)]
        \item $u_iv_j\in E(G)$ unless $i>r$ and $j>r$;
        \item $d_G(\{u_{i_1},\ldots,u_{i_r}\})\geq |V(F)|$ and $d_G(\{v_{i_1},\ldots,v_{i_r}\})\geq |V(F)|$ for $1\leq i_1<\ldots<i_r\leq r+m$.
    \end{enumerate}
    
    Let $\gamma=2\binom{r+m}{r}\cdot \binom{|V(F)|}{r}$. By Theorem \ref{supersatu}, there exists a constant $c=c_{r,r}(\gamma)$ such that any graph on $n$ vertices with $e>c\cdot n^{2-\frac{1}{r}}$ edges contains at least $\gamma \binom{n}{r}$ copies of $K_{r,r}$.
	
	Let $G$ be any graph with $e>c\cdot n^{2-\frac{1}{r}}$ edges. We assign an auxiliary graph $\mathcal{G}$ to $G$ as follows. The vertices of $\mathcal{G}$ are the $r$-sets in $V(G)$, and two such $r$-sets $U$ and $V$ are joined by an edge in $\mathcal{G}$ if $uv\in E(G)$ for every $u\in U$ and $v\in V$. Clearly, we have $\bar{d}(\mathcal{G})\geq 2\gamma$.
	
	Let us choose a uniformly random edge of $\mathcal{G}$ and let its endpoints be $X$ and $Y$ in uniformly random order. Observe that for any fixed $r$-set $U\in V(\mathcal{G})$, we have $\mathbb{P}(X=U)=\frac{d_{\mathcal{G}}(U)}{2e(\mathcal{G})}$. Let $u_1,\ldots,u_r$ be a uniformly random listing of the elements of $X$ and let $v_1,\ldots,v_r$ be a uniformly random listing of the elements of $Y$. If $d_G(X)\geq r+m$ and $d_G(Y)\geq r+m$, then let $v_{r+1},\ldots,v_{r+m}$ be chosen uniformly at random from $N_G(X)\setminus Y$ without repetition, and similarly, let $u_{r+1},\ldots,u_{r+m}$ be chosen uniformly at random from $N_G(Y)\setminus X$ without repetition (otherwise, let $v_{r+1},\ldots,v_{r+m},u_{r+1},\ldots,u_{r+m}$ be undefined).
	
	It is clear that if $d_G(X)\geq r+m$ and $d_G(Y)\geq r+m$, then these choices satisfy condition (i) above. It remains to be shown that with positive probability condition (ii) is also satisfied.
	
	But note that for any $1\leq i_1<\ldots<i_r\leq r+m$, the set $\{v_{i_1},\ldots,v_{i_r}\}$ is a uniformly random neighbour in $\mathcal{G}$ of $X$, where, as noted above, $\mathbb{P}(X=U)=\frac{d_{\mathcal{G}}(U)}{2e(\mathcal{G})}$. Hence,
	\begin{align}
		\mathbb{P}(\{v_{i_1},\ldots,v_{i_r}\}=V)&=\sum_{\substack{U\sim V \\ d_G(U)\geq r+m}} \mathbb{P}\big(X=U\big)\cdot \frac{1}{d_{\mathcal{G}}(U)} \nonumber \\
		&\leq \sum_{U\sim V} \mathbb{P}\big(X=U\big)\cdot \frac{1}{d_{\mathcal{G}}(U)} \nonumber \\
		&= \sum_{U\sim V} \frac{d_{\mathcal{G}}(U)}{2e(\mathcal{G})}\cdot \frac{1}{d_{\mathcal{G}}(U)} \nonumber \\ 
		&=\frac{d_{\mathcal{G}}(V)}{2e(\mathcal{G})}, \label{eqnstationary}
	\end{align} 
	where we write $U\sim V$ if $U$ and $V$ are neighbours in $\mathcal{G}$.
	
	Now let $\mathcal{S}$ consist of those $V\in V(\mathcal{G})$ for which $d_{\mathcal{G}}(V)\leq \frac{\bar{d}(\mathcal{G})}{4\binom{r+m}{r}}$. By inequality (\ref{eqnstationary}), for every $1\leq i_1<\ldots<i_r\leq r+m$, we have 
	$$\mathbb{P}(\{v_{i_1},\ldots,v_{i_r}\}\in \mathcal{S})\leq \frac{1}{2e(\mathcal{G})}\sum_{V\in \mathcal{S}} d_{\mathcal{G}}(V)\leq \frac{1}{4\binom{r+m}{r}}.$$
	Thus, with probability at least $3/4$, $\{v_{i_1},\ldots,v_{i_r}\}\not \in \mathcal{S}$ for every $1\leq i_1<\ldots<i_r\leq r+m$. Similarly, with probability at least $3/4$, $\{u_{i_1},\ldots,u_{i_r}\}\not \in \mathcal{S}$ holds for every $1\leq i_1<\ldots<i_r\leq r+m$. Hence, with probability at least $1/2$, we have both $\{u_{i_1},\ldots,u_{i_r}\}\not \in \mathcal{S}$ and $\{v_{i_1},\ldots,v_{i_r}\}\not \in \mathcal{S}$ for every $1\leq i_1<\ldots<i_r\leq r+m$. But if $U\not \in \mathcal{S}$, then $d_{\mathcal{G}}(U)>\frac{\gamma}{2\binom{r+m}{r}}\geq \binom{|V(F)|}{r}$. Therefore $d_G(U)\geq |V(F)|$ holds for all such $U$. It follows that with probability at least $1/2$, the vertices $u_1,\ldots,u_{r+m},v_1,\ldots,v_{r+m}$ are well-defined and have properties (i) and (ii).
\end{proof}

We now turn to the proof of Theorem \ref{main2}.

\begin{proof}[Proof of Theorem \ref{main2}]
	Let $k$ be the size of the $(r,t)$-blownup tree and let $\gamma=\frac{3}{2}k\cdot \binom{10k^2t^2}{r}$. By Theorem \ref{supersatu}, there exists a constant $c=c_{r,r}(\gamma)$ such that any graph on $n$ vertices with $e>c\cdot n^{2-\frac{1}{r}}$ edges contains at least $\gamma \binom{n}{r}$ copies of $K_{r,r}$.
	
	Let $G$ be any graph with $e>c\cdot n^{2-\frac{1}{r}}$ edges. Define the auxiliary graph $\mathcal{G}$ as in the proof of Theorem \ref{complexity2}. Clearly, we have $\bar{d}(\mathcal{G})\geq 2\gamma$.
	
	Let us define a random function $f$ which is a partial graph homomorphism $L\rightarrow G$, i.e., if it is defined on $S\subset V(L)$, then it is a graph homomorphism $L\lbrack S\rbrack \rightarrow G$.
	We define $f$ firstly on $X_1$, then on $Y_1$, $Y_2$, \ldots, and finally on $Y_k$.
	
	Let $f(X_1)$ be a random vertex of $\mathcal{G}$ according to the stationary distribution, that is, $f(X_1)=U$ with probability $\frac{d_{\mathcal{G}}(U)}{2e(\mathcal{G})}$. (Once $f(X_1)=U$ is decided, each bijection $X_1\rightarrow U$ is chosen with equal probability.) If $d_G(f(X_1))\geq t$, then let $f(Y_1)$ be a uniformly random $t$-subset of $N_G(f(X_1))$. Otherwise, let $f$ be undefined on $Y_1$.
	
	More generally, for $2\leq i\leq k$, choose $j<i$ such that $X_i\subset Y_j$. If $f$ is undefined on $Y_j$, then declare $f$ to be undefined on $Y_i$. Otherwise, let $U=f(X_i)$. If $d_G(U)<t$, then let $f$ be undefined on $Y_i$, while if $d_G(U)\geq t$, then let $f(Y_i)$ be a uniformly random $t$-subset of $N_G(U)$.
	
	It is clear that this produces a partial graph homomorphism $L\rightarrow G$.

	The key step in our proof is the following claim.
	
\begin{claim}For each $1\leq i\leq k$ and each $U\in V(\mathcal{G})$, $$\mathbb{P}(f(X_i)=U)\leq \frac{d_\mathcal{G}(U)}{2e(\mathcal{G})}.$$
\end{claim}
\begin{proof}
Fix $1\leq i\leq k$. Observe that there is a sequence $j_1<\ldots<j_\ell=i$ such that $X_{j_1}=X_1$ and for each $1\leq a\leq \ell-1$, we have $X_{j_{a+1}}\subset Y_{j_a}$. We prove by induction on $a$ that for each $1\leq a\leq \ell$ and every $U\in V(\mathcal{G})$, we have $\mathbb{P}(f(X_{j_a})=U)\leq \frac{d_\mathcal{G}(U)}{2e(\mathcal{G})}$. For $a=1$, we have $X_{j_a}=X_1$, so $\mathbb{P}(f(X_{j_a})=U)= \frac{d_\mathcal{G}(U)}{2e(\mathcal{G})}$. For $a\geq 2$, observe that conditional on $f(X_{j_{a-1}})=V$, $f(Y_{j_{a-1}})$ is defined if and only if $d_G(V)\geq t$, and if this holds, then $f(Y_{j_{a-1}})$ is a uniformly random $t$-set in $N_G(V)$. Therefore in this case $f(X_{j_a})$ is a uniformly random $r$-set in $N_G(V)$, so if $U\subset N_G(V)$ then the probability that $f(X_{j_a})=U$ is $\frac{1}{d_{\mathcal{G}}(V)}$. Hence, we have
	\begin{align*}
		\mathbb{P}(f(X_{j_a})=U)&=\sum_{\substack{V\sim U \\ d_G(V)\geq t}} \mathbb{P}\big(f(X_{j_{a-1}})=V\big)\cdot \frac{1}{d_{\mathcal{G}}(V)} \\
		&\leq \sum_{V\sim U} \mathbb{P}\big(f(X_{j_{a-1}})=V\big)\cdot \frac{1}{d_{\mathcal{G}}(V)} \\
		&\leq \sum_{V\sim U} \frac{d_{\mathcal{G}}(V)}{2e(\mathcal{G})}\cdot \frac{1}{d_{\mathcal{G}}(V)} \\
		&=\frac{d_{\mathcal{G}}(U)}{2e(\mathcal{G})},
	\end{align*}
	where we write $V\sim U$ if $U$ and $V$ are neighbours in $\mathcal{G}$.
	This completes the induction step, and the case $a=\ell$ proves the claim.
\end{proof}

	Now let $\mathcal{S}$ consist of those $U\in V(\mathcal{G})$ for which $d_{\mathcal{G}}(U)\leq \frac{\bar{d}(\mathcal{G})}{3k}$. By the claim above, for every $i$, we have $\mathbb{P}(f(X_i)\in \mathcal{S})\leq \frac{1}{2e(\mathcal{G})}\sum_{U\in \mathcal{S}} d_{\mathcal{G}}(U)\leq \frac{1}{3k}$. Thus, with probability at least $1/3$, $f(X_i)\not \in \mathcal{S}$ for every $i$. Moreover, for any $U\in V(\mathcal{G})\setminus \mathcal{S}$ we have $d_G(U)\geq t$, so if $f(X_i)\not \in \mathcal{S}$ for every $i$, then $f$ is defined everywhere.
	
	Suppose that $f(X_i)=U$ for some $U\in V(\mathcal{G})$ with $d_{\mathcal{G}}(U)>\frac{\bar{d}(\mathcal{G})}{3k}$. Then $d_{\mathcal{G}}(U)>\binom{10k^2t^2}{r}$, so $d_G(U)>10k^2t^2$. But $f(Y_i)$ is a uniformly random $t$-subset of $N_G(U)$, and $|f(\bigcup_{0\leq j\leq i-1} Y_j)|\leq kt$, so the probability that $f(Y_i)\cap f(\bigcup_{0\leq j\leq i-1} Y_j)\neq \emptyset$ is at most $\frac{1}{3k}$.
	
	It follows that with probability at least $1/3$, $f$ defines an injective graph homomorphism $L\rightarrow G$, thus $G$ contains $L$ as a subgraph.
\end{proof}

Given Theorem \ref{main2}, it is not hard to deduce Theorem \ref{main4}.
Clearly, it suffices to prove that any $r$-degenerate blow-up of a tree is a subgraph of some $(r,t)$-blownup tree. We will in fact prove the following stronger statement.
    
\begin{claim}
Let $F$ be a blow-up of some tree $T$, and suppose that $F$ is $r$-degenerate. For each $u\in V(T)$, write $I(u)$ for the independent set with which the vertex $u$ is replaced in $F$. Then there exists some $t=t(F)$ and an $(r,t)$-blownup tree $L$ with sets $X_1,\ldots,X_k$, $Y_0,\ldots,Y_k$ as in Definition~\ref{rtblownup} such that there is an embedding of $F$ in $L$ in a way that each $I(u)$ is a subset of some $Y_i$ for $0 \le i \le k$.
\end{claim}
\begin{proof}
The proof is by induction on the size of $T$. If $T$ has one vertex, the assertion is trivial. Now assume that $T$ has at least two vertices. The assertion is straightforward when $T$ is a star, so let us assume that that is not the case. Let $x$ be an arbitrary vertex of $T$ and let $u$ be a vertex with maximum distance from $x$. Clearly $u$ is a leaf. Let $v$ be the unique neighbour of $u$~in~$T$. Since $T$ is not a star, we have $v\neq x$.

If $|I(v)|\leq r$, then by induction there exist integers $t,k$ and an $(r,t)$-blownup tree $L$ with sets $X_1,\ldots,X_{k}$, $Y_0,\ldots,Y_k$ such that there is an embedding of $F-I(u)$ in $L$ in a way that for each $y\in V(T)\setminus \{u\}$, $I(y)$ is a subset of some $Y_i$. In particular, $I(v)$ is a subset of some $Y_i$, so we can take $X_{k+1}=I(v)$ and $Y_{k+1}=I(u)$ to get an embedding of $F$ in an $(r,t')$-blownup tree $L'$ of size $k+1$ with $t'=\max(t,|I(u)|)$.

We may therefore assume that $|I(v)|>r$. Then 
$$\sum_{w\in V(T):\ wv\in E(T)} |I(w)|\leq r,$$
for otherwise $F$ contains $K_{r+1,r+1}$ as a subgraph and so is not $r$-degenerate. Let $z$ be the unique neighbour of $v$ on the path between $v$ and $x$ and let $u_1,\ldots,u_m$ be the other neighbours of $v$. Now $T-\{v,u_1,\ldots,u_m\}$ is a tree, so by induction there exist integers $t,k$ and an $(r,t)$-blownup tree~$L$ with sets $X_1,\ldots,X_{k}$, $Y_0,\ldots,Y_k$ such that there is an embedding of $F-(I(v)\cup \bigcup_{j\leq m} I(u_j))$ in~$L$ in a way that for each $y\in V(T)\setminus \{v,u_1,\ldots,u_m\}$, $I(y)$ is a subset of some $Y_i$. In particular, $I(z)$ is a subset of some $Y_i$. Now if we replace $Y_i$ with $Y'_i=Y_i\cup \bigcup_{j\leq m} I(u_j)$ and set $X_{k+1}=I(z)\cup \bigcup_{j\leq m} I(u_j)\subset Y'_i$ and $Y_{k+1}=I(v)$, then we get an embedding of $F$ in an $(r,t')$-blownup tree $L'$ of size $k+1$ with $t'=\max(t,|Y'_i|,|I(v)|)$.
\end{proof}

\section{Concluding remarks and open problems} \label{sectionfurtherblowup}

 In this paper we were focusing on the extremal number of blow-ups of trees, but it is natural to study the extremal number of the blow-ups of arbitrary graphs. We make the following conjecture, relating the Turán number of a bipartite graph $F$ and that of its blow-up $F[r]$. % Here the next step could be to consider the blow-ups of even cycles. We conjecture that this dependence relies only on the edge density of the auxiliary graph defined above, i.e. if the ratio of the number of edges in the auxiliary graph $\mathcal{G}$ and the Turán number $\ex(|V(\mathcal{G}|, F)$ of $F$ attains an appropriate constant (depending only on $r$ and $F$), then one can always embed $F$ to  $\mathcal{G}$ such that the $r$-sets  corresponding to the vertices of $F$ are disjoint. This would in turn provides an embedding of the blow-up $F[r]$ to $G$ as well.
 
 \begin{conj}\label{con} For any $0\leq \alpha\leq 1$ and any graph $F$, if $\ex(n, F)= O(n^{2-\alpha})$, then  $$\ex(n,F[r])= O\left(n^{2-\frac{\alpha}{r}}\right).$$
 \end{conj}

 Note that if the number of edges in $G$ is $\omega\left(n^{2-\frac{\alpha}{r}}\right)$, then by supersaturation there are $\omega\left(N^{2-\alpha}\right)$ edges  in the auxiliary graph $\mathcal{G}$, where $N=|V(\mathcal{G})|=\binom{n}{r}$.
 Therefore there exists a copy of $F$ in $\mathcal{G}$, which provides a homomorphic copy of $F[r]$ in $G$. We conjecture that one can always embed $F$ to  $\mathcal{G}$ in a way that the $r$-sets  corresponding to the vertices of $F$ are disjoint, providing an embedding of the blow-up $F[r]$ to $G$. 
 
 We have proved Conjecture \ref{con} for trees. Note that $K_{s,t}[r]=K_{rs,rt}$, so the conjecture also holds for $F=K_{s,t}$, $\alpha=\frac{1}{s}$. It would be interesting to extend this to the family of even cycles. 
 
 %Since $C_4[r]=K_{2r,2r}$, the conjecture holds for $F=C_4$.

%We mention that our methods enable us to prove similar results in yet a larger family of blow-ups of trees compared to Theorem \ref{main2}. First observe that we might allow non-leaves $v_i$ to have $|I(v_i)|>r$ is certain cases, namely when one can embed these graphs to blow-up graphs described in Theorem \ref{main2}, see Figure \ref{Fig:folding}. Secondly, the blow-up class of one non-leaf can always have larger cardinality then $r$ if we replace the first step of our proof by the application of the dependent random choice method, which provides an arbitrary large set of constant size  $r^*$ where  all the $r$-sets have  at least $\binom{rk+r^*}{r}$ common neighbours. We leave the details to the interested reader.


\begin{thebibliography}{}

\bibitem{Alon}
Alon, N., Krivelevich, M.,  Sudakov, B. (2003). Turán numbers of bipartite graphs and related Ramsey-type questions. Combinatorics, Probability and Computing, 12(5--6), 477--494.
\bibitem{ARSz}
Alon, N., Rónyai, L.,  Szabó, T. (1999). Norm-graphs: variations and applications. Journal of Combinatorial Theory, Series B, 76(2), 280--290.


%\bibitem{Bukh}
%Bukh, B.,  Tait, M. (2018). Tur\'an number of theta graphs. arXiv preprint arXiv:1804.10014.

%\bibitem{CKLL18}
%Conlon, D., Kim, J.H., Lee, C. and Lee, J. (2018). Some advances on Sidorenko's conjecture. Journal of the London Mathematical Society, 98(3), 593--608.

%\bibitem{Conlon2}
%Conlon, D., Kim, J. H., Lee, C., Lee, J. (2018). Some advances on Sidorenko's conjecture. Journal of the London Mathematical Society, 98(3), 593--608.

\bibitem{Conlon}
Conlon, D.,  Lee, J. (2018). On the extremal number of subdivisions. arXiv preprint arXiv:1807.05008.


\bibitem{Erdos}
Erdős, P. (1967). Some recent results on extremal problems in graph theory. Results. Theory of Graphs (Internat. Sympos., Rome, 1966), 117--123.

\bibitem{Erdos-Simonovits}
Erdős, P.,  Simonovits, M. (1983). Supersaturated graphs and hypergraphs. Combinatorica, 3(2), 181--192

\bibitem{E-Si} Erdős, P.,  Simonovits, M. (1982). Compactness results in extremal graph theory. Combinatorica, 2(3), 275--288.

\bibitem{Fox}
Fox, J.,  Sudakov, B. (2011). Dependent random choice. Random Structures \& Algorithms, 38(1--2), 68--99.

\bibitem{Furedi} Füredi, Z. (1991). On a Turán type problem of Erdős. Combinatorica, 11(1), 75--79.

\bibitem{Furedi-cube} Füredi, Z. (2013). On a theorem of Erd\H {o}s and Simonovits on graphs not containing the cube. arXiv preprint arXiv:1307.1062.

\bibitem{Furedi-Simonovits} Füredi, Z.,  Simonovits, M. (2013). The history of degenerate (bipartite) extremal graph problems. In Erdős Centennial (pp. 169--264). Springer, Berlin, Heidelberg.

\bibitem{Furedi-West}
Füredi, Z., West, D. B. (2001). Ramsey theory and bandwidth of graphs. Graphs and Combinatorics, 17(3), 463--471.

\bibitem{JanzerO}
Janzer, O. (2018). Improved bounds for the extremal number of subdivisions. arXiv preprint arXiv:1809.00468.

%\bibitem{Kang}
%Kang, D. Y., Kim, J.,  Liu, H. (2018). On the rational Tur\'an exponents conjecture. arXiv preprint arXiv:1811.06916.

\bibitem{KRSz} Kollár, J., Rónyai, L.,  Szabó, T. (1996). Norm-graphs and bipartite Turán numbers. Combinatorica, 16(3), 399--406.

\bibitem{KST} K\H{o}v\'ari, T., S\'os, V., Tur\'an, P. (1954). On a problem of K. Zarankiewicz, Colloquium Math. 3, 50--57.

\bibitem{ZNagy} Nagy, Z. L. (2019). Supersaturation of $C_4$: From Zarankiewicz towards Erdős--Simonovits--Sidorenko. European Journal of Combinatorics, 75, 19--31.

\bibitem{Sim84} Simonovits, M. (1984). Extremal graph problems, degenerate extremal problems, and supersaturated graphs. Progress in graph theory (Waterloo, Ont., 1982), Academic Press, Toronto, ON, 419--437.

\end{thebibliography}
\end{document}